\documentclass[11pt]{article}
\usepackage{fullpage,amsthm,amssymb,tikz,bm,verbatim}
\usepackage[font=footnotesize]{caption} 
\usepackage{moresize}
\def\vph{\varphi}

\def\VV{\mathcal{V}}
\def\HH{\mathcal{H}}
\def\EE{\mathcal{E}}
\def\G{\mathcal{G}}

\def\chio{\chi_o}
\def\mad{\mathrm{mad}}
\def\Mad{\mathrm{Mad}}

\newtheorem{thmA}{Theorem}
\newtheorem*{thmTa}{Theorem 2(a)}
\newtheorem*{thmTb}{Theorem 2(b)}
\newtheorem{lem}{Lemma}
\newtheorem{prop}[lem]{Proposition}
\newtheorem{thm}{Theorem}
\newtheorem{conj}{Conjecture}
\newtheorem{cor}[lem]{Corollary}
\theoremstyle{definition}
\newtheorem{rem}[lem]{Remark}

\RequirePackage{marginnote,hyperref}
\newcommand{\aside}[1]{\marginnote{\scriptsize{#1}}[0cm]}

\newcommand\Emph[1]{\emph{#1}\aside{#1}}

\author{Daniel W. Cranston\thanks{%
Department of Computer Science, Virginia Commonwealth
University, Richmond, VA, USA;
\texttt{dcranston@vcu.edu}
}}
\begin{document}
\title{Odd Colorings of Sparse Graphs}
\maketitle
\abstract{
A proper coloring of a graph is called \emph{odd} if every
non-isolated vertex has some color that appears an odd number of times on its
neighborhood.  The smallest number of colors that admits an odd coloring of a
graph $G$ is denoted $\chio(G)$.  
This notion was introduced by Petru\v{s}evski and \v{S}krekovski, who
proved that if $G$ is planar then $\chio(G)\le 9$; they also conjectured that
$\chio(G)\le 5$.  
For a positive real number $\alpha$, we consider the
maximum value of $\chio(G)$ over all graphs $G$ with maximum average degree 
less than $\alpha$; we denote this value by $\chio(\G_{\alpha})$.  
We note that $\chio(\G_{\alpha})$ is undefined for all $\alpha\ge 4$.  In
contrast, for each $\alpha\in[0,4)$, we give a (nearly sharp) upper bound on
$\chio(\G_{\alpha})$.  
Finally, we 
prove $\chio(\G_{20/7})= 5$ and 
$\chio(\G_3)= 6$.  Both of these results are sharp.

}
\bigskip

\section{Introduction}

A proper coloring of a graph is \emph{odd}\aside{odd coloring} if every
non-isolated vertex has some color that appears an odd number of times on its
neighborhood.  The smallest number of colors that admits an odd coloring of a
graph $G$ is denoted \Emph{$\chio(G)$}.  
Clearly, $\chio(G)\le |V(G)|$, since we can simply color each vertex with its own color. 
This notion was introduced by
Petru\v{s}evski and \v{S}krekovski~\cite{PS}, who 
proved that if $G$ is planar then
$\chio(G)\le 9$; they also conjectured that $\chio(G)\le 5$.

Odd coloring is motivated by various types of hypergraph coloring.
A hypergraph $\HH$ consists of a set $\VV$ of vertices and a set $\EE$ of
(hyper)edges, each of which consists of an arbitrary set of vertices in $\VV$.
Most varieties of hypergraph coloring assign colors (integers in $\{1,\ldots,k\}$)
to the elements of $\VV$ subject to certain constraints.
Standard hypergraph coloring requires only that no edge in $\EE$ is
monochromatic.  Even et al.~\cite{ELRS} introduced \emph{conflict-free} coloring,
which requires that each edge in $\EE$ has some color that appears exactly once on
its vertices.  This topic has been widely
studied~\cite{CT,GST,KKL,PT,smorodinsky}.
Cheilaris et al.~\cite{CKP} studied \emph{odd} coloring of hypergraphs, which
requires that each edge in $\EE$ has some color that appears an odd number of times
on its vertices.  Aspects of this problem have been studied in~\cite{BMWW1,
BMWW2, FG}. For graphs, Cheilaris~\cite{cheilaris-diss} studied conflict-free
colorings with respect to open neighborhoods.  That is, for each vertex $v$
some color appears exactly once on $N(v)$.  Finally, Petru\v{s}evski and
\v{S}krekovski~\cite{PS} studied odd colorings of
graphs, which are proper colorings where each vertex $v$ has some color that appears
an odd number of times on $N(v)$.  It is this parameter that we consider in the
present short note.

The average degree of a graph $H$ is $2|E(H)|/|V(H)|$.  The \emph{maximum
average degree} of a graph $G$, denoted \Emph{$\mad(G)$}, is the maximum, over all
non-empty subgraphs $H$ of $G$, of the average degree of $H$.  That is,
$\mad(G):=\max_{H\subseteq G}2|E(H)|/|V(H)|$.
For each positive real number $\alpha$, let $\G_{\alpha}$ denote the family of
graphs $G$ with $\mad(G)<\alpha$.
We denote by $\chio(\G_{\alpha})$\aside{$\G_{\alpha}$, $\chio(\G_{\alpha})$} the 
maximum value of $\chio(G)$ over all $G\in \G_{\alpha}$.
The focus of this paper is bounding $\chio(\G_{\alpha})$ for various values of
$\alpha$.
We observe that $\chio(\G_{\alpha})$ is undefined for all $\alpha\ge 4$.  
That is, there exists a sequence of graphs $G_n$ such that $\chio(G_n)=n$ and
$\mad(G_n)<4$ for all $n$.
In contrast, for each $\alpha\in[0,4)$, we give a (nearly sharp) upper bound on
$\chio(\G_{\alpha})$.  We have two main results.

\begin{thmA}
Fix $\epsilon$ such that $0<\epsilon\le 8/5$.
If $\mad(G)\le 4-\epsilon$, then $\chio(G)\le \lfloor 8/\epsilon\rfloor+2$.
As $\epsilon\to 0$, infinitely often there exists $G_{\epsilon}$ such that
$\mad(G_{\epsilon})=4-\epsilon$ and $\chio(G_{\epsilon})=\lfloor
8/\epsilon\rfloor-1$.
\end{thmA}

When $\epsilon\in\{1,8/7\}$, we prove sharper upper bounds on
$\chio(G_{4-\epsilon})$.

\begin{thmA}
Fix a graph $G$.  (a) If $\mad(G)<3$, then $\chio(G)\le 6$;
and (b) if $\mad(G)<20/7$, then $\chio(G)\le 5$. 
Furthermore, neither of these upper bounds on $\chio(G)$ can be decreased, and
neither of the inequalities can be weakened to allow equality.
\end{thmA}

For a proper coloring $\vph$ of some subgraph $H$ of $G$, let \Emph{$\vph_o(v)$} 
denote the unique color that appears an odd number of times in $N_H(v)$ if such
a color exists; otherwise, $\vph_o(v)$ is undefined.  Most of the rest of our
notation and definitions are standard.  But for the reader's convenience we
highlight a few terms.  The \Emph{girth} of a graph $G$ is the length of a
shortest cycle in $G$ (the girth of an acyclic graph is infinite).  A
$k$-vertex\aside{$k$/$k^+$/$k^-$-vertex} is a vertex of degree $k$.  A
$k^+$-vertex (resp.~$k^-$-vertex) is a vertex of degree at least (resp.~at
most) $k$.  A $k$-neighbor\aside{$k$/$k^+$/$k^-$-neighbor} of a vertex $v$ is a
neighbor of $v$ that is a $k$-vertex.  Both $k^+$-neighbor and $k^-$-neighbor
are defined analogously.

To close this introduction, we prove three easy results about graphs $G$ with
$\chio(G)\le 4$.

\begin{prop}
A graph $G$ has $\chio(G)=1$ if and only if $G$ has vertices but no edges.
And $G$ has $\chio(G)=2$ if and only if $G$ is bipartite
and the degree of each vertex in $G$ is either 0 or odd.
\end{prop}
\begin{proof}
The first statement is obvious, so now we prove the second.
If $G$ is bipartite, then $G$ has a proper 2-coloring $\vph$.  If the degree of
each vertex is either 0 or odd, then $\vph$ is also an odd 2-coloring of $G$.
If $G$ is not bipartite, then $G$ has no proper 2-coloring, so clearly
$\chio(G)>2$.  Suppose instead $G$ is bipartite, but some vertex $v$ has
positive even degree.  Now the component of $G$ containing $v$ has only a single
2-coloring (up to permuting colors), but this 2-coloring is not odd since
 $N(v)$ has only a single color, which appears an even number of
times.
\end{proof}

\begin{prop}
If $3|n$, then $\chio(C_n)=3$.  If $n=5$, then $\chio(C_n)=5$.  Otherwise, $\chio(C_n)=4$.
\label{cycle-prop}
\end{prop}
\begin{proof}
If $3|n$, then we can repeat the colors $1,2,3,1,2,3,\ldots$ around $C_n$ to get an odd
3-coloring.  Suppose instead that $3\nmid n$.  By the Pigeonhole principle, in every
proper 3-coloring $\vph$ some color appears twice on the neighborhood
of some vertex $v$, so $\vph$ is not odd.  Thus, $\chio(C_n)\ge 4$.
If $3\nmid n$ and $n\ne 5$, then we can begin with $1,2,3,4$ or
$1,2,3,4,1,2,3,4$ and continue with $1,2,3,1,2,3\ldots$.  Finally, it is easy to
check that each proper 4-coloring $\vph$ of $C_5$ uses the same color on both neighbors
of some vertex, so $\vph$ is not an odd coloring.  Thus, $\chio(C_5)=5$.
\end{proof}

It is interesting to note that the class of graphs $G$ with $\mad(G)=2$ and
$\chio(G)=4$ is richer than simply cycles $C_n$ with $3\nmid n$.  
Denote the vertices of $C_n$ by $v_1,\ldots,v_n$.  Starting from such a graph,
we can add arbitrarily many leaves at each $v_i$ with $3|i$.
It is straightforward to check that the resulting graph $G'$ also has
$\mad(G')=2$ and $\chio(G')=4$.  More generally, we can identify these $v_i$
with vertices of an arbitrary graph and $\chio$ will not decrease.

\begin{prop}
Every tree $T$ has $\chio(T)\le 3$.  Thus, if $\mad(T)<2$, then $\chio(T)\le
3$.  This bound on $\mad(T)$ is sharp.
\label{tree-prop}
\end{prop}
\begin{proof}
Let $T$ be a tree.
We use induction on $|T|$.  The case $|T|=1$ is trivial.
Now suppose that $|T|\ge 2$.  
Let $v$ be a leaf of $T$.
Let $T':=T-v$ and note that $T'$ is a
tree.  By hypothesis, $T'$ has an odd 3-coloring $\vph$.
Let $w$ be the neighbor of $v$ in $T$.  
To extend $\vph$ to an odd 3-coloring of $T$, we color $v$ with a color
outside $\{\vph(w),\vph_o(w)\}$.  This gives the desired odd 3-coloring of $T$.

If $\mad(G)<2$, then $G$ is a forest.  So the second statement follows from the
first.  Finally, the third statement follows from Proposition~\ref{cycle-prop}.
\end{proof}

The following proposition is folklore.  But we include a proof for completeness.

\begin{prop}
If $G$ is a planar graph with girth at least $g$, then $\mad(G)<2g/(g-2)$.
\label{girth-prop}
\end{prop}
\begin{proof}
Fix a plane embedding of a planar graph $G$ with girth at least $g$.  Since each
subgraph of $G$ also has girth at least $g$, it suffices to show that
$2|E(G)|/|V(G)|<2g/(g-2)$.  Summing the lengths of all facial walks gives
$2|E(G)|\ge g|F|$, where $F$ is the set of all faces.  
To prove the proposition,
we substitute this inequality into Euler's formula and solve for
$2|E(G)|/|V(G)|$.
\end{proof}

\section{$\Mad(G)<4$}
Recall, for each $\alpha>0$, that $\G_{\alpha}$ denotes the set of all graphs $G$
with $\mad(G)<\alpha$; and we write $\chio(\G_{\alpha})$ to denote the maximum
value of $\chio(G)$ over all $G\in \G_{\alpha}$.  Proposition~\ref{tree-prop}
shows that $\chio(G_2)=3$.  In this section, we determine the set of all values
$\alpha$ such that $\chio(\G_{\alpha})$ is defined; when it is defined, we prove
a (nearly sharp) upper bound on its value.

We denote by \Emph{$K_n^*$} the graph formed from $K_n$ by subdiving each edge once. 


\begin{lem}
We have $\chio(K_n^*)=n$ and $\mad(K_n^*)=4-8/(n+1)$.
\label{lem1}
\label{mad4lem}
\end{lem}
\begin{proof}
For each $n\ge 1$, let $\epsilon_n:=8/(n+1)$.
To show that $\mad(K^*_n)= 4-\epsilon_n$,
we give a fractional orientation of $K^*_n$ where each vertex has indegree exactly
$2-\epsilon_n/2$.  We orient each edge with fraction $1-\epsilon_n/4$ towards its
endpoint of degree 2 and fraction $\epsilon_n/4$ toward its endpoint of degree
$n$.  Each 2-vertex has indegree $2(1-\epsilon_n/4)=2-\epsilon_n/2$.
Each $(n-1)$-vertex has indegree $(n-1)(\epsilon_n/4) = (n-1)(2/(n+1)) =
(2n-2)/(n+1) = 2-4/(n+1) = 2-\epsilon_n/2$.  Thus, $\mad(G)=4-\epsilon_n$, as
desired.

In an odd coloring of $K^*_n$, all $(n-1)$-vertices must get distinct colors.  So
$\chio(K^*_n)\ge n$.  Given any coloring where all $(n-1)$-vertices get distinct
colors, it is easy to extend to an odd $n$-coloring of $K^*_n$.  
Thus, as claimed, $\chio(K^*_n)=n$.
\end{proof}

\begin{cor}
There exists a sequence $\epsilon_1,\epsilon_2,\ldots$ such that
$\epsilon_n>0$ for all $n\ge 1$ and $\lim_{n\to\infty}\epsilon_n=0$ and for each
$n\ge 1$ there exists a graph $G_n$ with $\mad(G_n)=4-\epsilon_n$ and
$\chio(G_n)=8/\epsilon_n-1$.
\label{cor1}
\end{cor}
\begin{proof}
Let $\epsilon_n:=8/(n+1)$ and $G_n:=K^*_n$.
Now $\chio(G_n)=n=(n+1)-1=8/\epsilon_n-1$.
\end{proof}

\begin{cor}
$\chio(\G_{\alpha})$ is undefined whenever $\alpha\ge 4$.
\label{cor2}
\end{cor}

In Lemma~\ref{lem1}, we considered $K^*_n$ which is formed by subdividing each
edge of $K_n$.  Applying the same construction to any $n$-chromatic graph $H$
yields a graph $H'$ with $\chio(H')=n$ and $\mad(H')<4$.  
Since there exist graphs $H$ with both chromatic number and girth arbitrarily
large, there also exist graphs $H'$ with $\chio(H')$ and girth arbitrarily
large, and with $\mad(H')<4$.
However, every $n$-chromatic graph $H$
that does not contain $K_n$ as a subgraph gives an upper bound on
$4-\epsilon$ that is worse (larger) than that in Corollary~\ref{cor1}.
Consider
an $n$-critical (sub)graph $H$ with $a$ vertices and $b$ edges.
Recall that $\delta(H)\ge n-1$.  But $H$ is not $(n-1)$-regular,
by Brooks' Theorem.  Thus, $b=|E(H)|>(n-1)|V(H)|/2=(n-1)a/2$.
Subdividing each edge of $H$ gives $H'$ with $|V(H')|=a+b$ and
$|E(H')|=2b>a(n-1)$.  Thus, $\mad(H')\ge 2|E(H')|/|V(H')| =
4b/(a+b)>2a(n-1)/(a(n-1)/2+a)=2a(n-1)/(a(n+1)/2)=4(n-1)/(n+1)=4-8/(n+1)$.

We next show that the construction $K_n^*$ in Lemma~\ref{mad4lem} is
nearly sharp.

\begin{thm}
Fix $\epsilon$ such that $0<\epsilon\le 8/5$.
If $\mad(G)\le 4-\epsilon$, then $\chio(G)\le \lfloor 8/\epsilon\rfloor+2$.
As $\epsilon\to 0$, infinitely often there exists $G_{\epsilon}$ such that
$\mad(G_{\epsilon})=4-\epsilon$ and $\chio(G_{\epsilon})=\lfloor
8/\epsilon\rfloor-1$.
\label{mad4thm}
\end{thm}

\begin{proof}
The second statement follows directly from Corollary~\ref{cor1}.

For the first statement, let $k:=\lfloor 8/\epsilon\rfloor+2$\aside{$k$} and
note that $k\ge 7$.  Our proof is by induction on $|V(G)|$.  The base case is
when $|V(G)|=1$, so $\chio(G)=1$.
%
Instead assume $|V(G)|\ge 2$.
If $G$ contains a 1-vertex $v$ with
neighbor $w$, then $G-v$ has an odd $k$-coloring $\vph$ and we can extend
$\vph$ to $v$ by coloring $v$ to avoid $\vph(w)$ and $\vph_o(w)$.
If $G$ contains a 3-vertex $v$, then denote $N(v)$ by $\{w_1,w_2,w_3\}$.
Now $G-v$ has an odd $k$-coloring $\vph$ and we can extend $\vph$ to $v$ by
coloring $v$ to avoid
$\{\vph(w_1), \vph(w_2), \vph(w_3), \vph_o(w_1), \vph_o(w_2), \vph_o(w_3)\}$.
So we assume that $G$ has neither 1-vertices nor 3-vertices.  
Suppose that $G$ has adjacent 2-vertices $v_1$ and $v_2$, and denote the
remaining neighbors of $v_1$ and $v_2$, respectively, by $v_0$ and $v_3$.
Now $G-\{v_1,v_2\}$ has an odd $k$-coloring $\vph$.  To extend this to $v_1$
and $v_2$, we first color $v_1$ to avoid
$\{\vph(v_0),\vph_o(v_0),\vph(v_3)\}$ and then color $v_2$ to avoid the new
color on $v_1$ as well as $\{\vph(v_0),\vph(v_3),\vph_o(v_3)\}$.

Since $\mad(G)<4$, we know that $\delta(G)=2$.  
We use discharging to show that some vertex $v$
with $d(v)\le 8/\epsilon-2$ has ``many'' 2-neighbors.  
Similar to the case of adjacent 2-vertices above, we will delete $v$ and its
2-neighbors, find an odd $k$-coloring $\vph$ for this smaller graph, and
extend $\vph$ to all of $G$.
Let $x:=1-\epsilon/2$\aside{$x$}.  We give each vertex $v$ initial charge
$d(v)$ and use the following single discharging rule: Each 2-vertex takes
charge $x$ from each neighbor.  Since $G$ does not have adjacent 2-vertices,
each 2-vertex finishes with charge $2+2x=4-\epsilon$.  For each $4^+$-vertex
$v$, let \Emph{$d_2(v)$} denote the number of 2-neighbors of $v$.  Now $v$ finishes
with charge $d(v)-xd_2(v)$.  If $d(v)-xd_2(v)>2+2x=4-\epsilon$ for each $v$,
then we contradict the assumption $\mad(G)\le 4-\epsilon$.  So there exists $v$
such that $d(v)-xd_2(v)\le 2+2x$.

Form $G'$ from $G$ by deleting $v$ and all of its 2-neighbors.  By induction,
$G'$ has an odd $k$-coloring $\vph$.  To extend $\vph$ to $G$, we first color
$v$ to avoid the colors on the colored neighbors (in $G$) of its deleted
2-neighbors (in $G'$) as well as, for each $4^+$-neighbor $w$, to avoid
$\vph(w)$ and $\vph_o(w)$.  For convenience, we denote this color used on $v$
by $\vph(v)$.
Then we color each 2-neighbor $x$ of $v$, with
other neighbor $y$, to avoid $\{\vph(v),\vph_o(v),\vph(y),\vph_o(y)\}$.
This gives an odd coloring $\vph'$ of $G$.  We must only show that $\vph'$ uses
at most $k$ colors.  In total, $v$ must avoid at most $2d(v)-d_2(v)$
colors.  So it will suffice to show that $k\ge 1+(2d(v)-d_2(v))$.
Note that $1+2d(v)-d_2(v)< 1+d(v)+(d(v)-xd_2(v))\le 1+d(v)+2+2x<1+d(v)+4$.
Since $1+2d(v)-d_2(v)$ is an integer, $1+2d(v)-d_2(v)\le d(v)+4$.
So now we must bound $d(v)$.  Clearly, $d(v)-d(v)x\le d(v)-d_2(v)x\le 2+2x$.
So $d(v)\le (2+2x)/(1-x) = (4-\epsilon)/(\epsilon/2) = 8/\epsilon-2$.
Since $d(v)$ is an integer, $d(v)\le \lfloor 8/\epsilon\rfloor-2$.
Thus, $d(v)+4\le (\lfloor 8/\epsilon\rfloor-2)+4=\lfloor 8/\epsilon\rfloor+2=
k$.
\end{proof}

\begin{rem}
In the previous proof we handled adjacent 2-vertices as follows.
``Suppose that $G$ has adjacent 2-vertices $v_1$ and $v_2$, and denote the
remaining neighbors of $v_1$ and $v_2$, respectively, by $v_0$ and $v_3$.
Now $G-\{v_1,v_2\}$ has an odd $k$-coloring $\vph$.  To extend this to $v_1$
and $v_2$, we first color $v_1$ to avoid
$\{\vph(v_0),\vph_o(v_0),\vph(v_3)\}$ and then color $v_2$ to avoid the new
color on $v_1$ as well as $\{\vph(v_0),\vph(v_3),\vph_o(v_3)\}$.''  In this
argument we implicitly assumed that $v_3\ne v_0$.  Otherwise, our choice for
color on $v_1$ might create a problem for $v_3$.  Specifically, suppose that
$v_0$ is adjacent to 2-vertices $v_1$ and $v_2$, which are also adjacent to each
other, and that colors 1 and 2 each appear once on $N(v)$ under $\vph$ (and no
other color appears an odd number of times); this is possible when $v_3=v_0$.
Now $\vph_o(v_0)$ is undefined.  But if we are careless and color $v_1$ with 1
and color $v_2$ with 2, then we create a problem for $v_0$.  This obstacle has
an easy work-around.  Rather than defining $\vph_o(v_0)$ relative to the coloring
$\vph$, we define it relative to the \emph{current coloring}, including any
vertices that were deleted but are now already colored.  For brevity, we omit
mention of this issue throughout the paper, since the solution above always
works.
\end{rem}

It is natural to consider a list-coloring analogue of odd coloring (although we
omit a formal definition).  We note that, with minor modifications, the
proof of Theorem~\ref{mad4thm} works equally well in the context of odd
list-coloring and even in the context of odd correspondence coloring.
The same is true for the proofs of Theorem~2(a) and Theorem~2(b).

We suspect that the construction in Lemma~\ref{lem1} is sharp.

\begin{conj}
Fix $\epsilon$ such that $0<\epsilon\le 8/5$.
If $\mad(G)\le 4-\epsilon$, then $\chio(G)\le \lfloor 8/\epsilon\rfloor-1$.
\label{conj1}
\end{conj}

Conjecture~\ref{conj1} cannot be extended to allow $\epsilon=2$, since
$\mad(C_{3s+1})=2$, but $\chio(C_{3s+1})=4$, for each positive integer $s$.  In
contrast, when $\epsilon=1$ and $\epsilon=8/7$, we prove the conjecture in a
stronger form.  We show that if $\mad(G)\le4-\epsilon$, then
$\chio(G)\le\lfloor 8/\epsilon\rfloor-2$ unless $G$ contains as a subgraph
$K^*_s$, where $s:=\lfloor 8/\epsilon\rfloor -1$.  We present these proofs in
Sections~\ref{6coloring-sec}~and~\ref{5coloring-sec}. 

Recall that a graph $H$ is \Emph{$k$-critical} if $\chi(H)=k$ and $\chi(H-e)<k$ for
every edge $e\in E(H)$.  Analogously, a graph $H$ is \Emph{odd-$k$-critical} if
$\chio(H)=k$ and $\chio(H-e)<k$ for every edge $e\in E(H)$.  The latter notion
is more subtle\footnote{An interesting example is $C_5$, since $\chio(C_5)=5$,
but $C_5$ has no odd-4-critical subgraph.}, since possibly $H'$ is a subgraph
of $H$, but $\chio(H')> \chio(H)$.  More concretely, if we form $H$ from
$K^*_n$ with $n\ge 3$ by adding a leaf adjacent to each vertex, then
$\chio(H)\le3$, even though $\chio(K^*_n)=n$.  It is easy to check that
subdividing every edge of a $k$-critical graph, with $k\ge 6$, yields an
odd-$k$-critical graph.  Kostochka and Yancey~\cite{KY1,KY2} showed that the
sparsest $k$-critical graphs are so-called $k$-Ore graphs.  For each $k\ge 6$,
by subdividing all edges in $k$-Ore graphs, we get an infinite family of
odd-$k$-critical graphs with maximum average degrees slightly larger than
$4-8/(k+3)$.  We suspect that in the limit no infinite family of 
odd-$k$-critical graphs has smaller maximum average degree.

\section{Odd 6-Colorings}
\label{6coloring-sec}
The goal of this section is to prove the following result.

\begin{thmTa}
\label{6color-thm}
If $\mad(G)<3$, then $\chio(G)\le 6$.  
This includes all planar $G$ of girth at least 6.
\end{thmTa}

Note that Theorem~\ref{6color-thm} is sharp in two senses.
First, $\mad(K_6^*)=4-8/(6+1)=2+6/7$ and $\chio(K_6^*)=6$, so the upper bound
of 6 cannot be decreased.  Second, $\mad(K_7^*)=4-8/(7+1)=3$ and
$\chio(K_7^*)=7$, so the hypothesis $\mad(G)<3$ cannot be weakened at all.

\begin{proof}
The second statement follows from the first by Proposition~\ref{girth-prop}.
So we prove the first.  Assume the statement is false, and $G$ is a
counterexample with as few vertices as possible.  As in the proof of
Theorem~\ref{mad4thm}, we assume 
$G$ has no 1-vertex and $G$ has no adjacent 2-vertices.
Suppose $G$ has a 3-vertex $v$ with neighbors $w_1,w_2,w_3$, where
$d(w_1)=2$.  Denote the other neighbor of $w_1$ by $x_1$.  Let $G':=G-\{v,w_1\}$.
 By minimality, $G'$ has an odd 6-coloring $\vph$.  To extend $\vph$ to $G$,
we color $v$ to avoid $\{\vph(w_1), \vph(w_2), \vph(w_3), \vph_o(w_2),
\vph_o(w_3)\}$ and then color $w_1$ to avoid the new color on $v$ and also
to avoid $\{\vph(x_1),\vph_o(x_1)\}$.  Thus, no 3-vertex has a 2-neighbor.

Suppose $G$ has a 4-vertex with neighbors $w_1,w_2,w_3,w_4$, where
$d(w_1)=d(w_2)=d(w_3)=2$.  For each $i\in \{1,2,3\}$, denote the other neighbor of
$w_i$ by $x_i$.  Let $G':=G-\{v,w_1,w_2,w_3\}$.  By minimality, $G'$ has an odd
6-coloring $\vph$.  To extend $\vph$ to $G$, color $v$ to avoid
$\{\vph(x_1), \vph(x_2), \vph(x_3)$, $\vph(w_4), \vph_o(w_4)\}$ and then
color each $w_i$ to avoid the new color on $v$ and also avoid
$\{\vph(x_i)$, $\vph_o(x_i), \vph(w_4)\}$.  Thus, no 4-vertex in $G$ has
three (or more) 2-neighbors.

Suppose $G$ has a 5-vertex $v$ with five 2-neighbors $w_1, w_2, w_3, w_4, w_5$.
 For each $i\in\{1,\ldots,5\}$, denote the other neighbor of $w_i$ by $x_i$. 
Let $G':=G-\{v,w_1,w_2,w_3,w_4,w_5\}$.  By minimality, $G'$ has an odd
6-coloring $\vph$.  To extend $\vph$ to $G$, we color $v$ to avoid
$\{\vph(x_1),\vph(x_2),\vph(x_3),\vph(x_4),\vph(x_5)\}$.  We then color
each $w_i$ to avoid the new color on $v$, as well as
$\{\vph(x_i),\vph_o(x_i),\vph(v)\}$.  Thus, no 5-vertex in $G$ has five
2-neighbors.

Now we use discharging to reach a contradiction, which will finish the proof.
We give each vertex $v$ initial charge $d(v)$ and use the following single
discharging rule: Each 2-vertex takes $1/2$ from each neighbor.  We show that
each vertex finishes with charge at least 3, which contradicts the hypothesis
$\mad(G)<3$.  Recall that $\delta(G)\ge 2$.

If $d(v)=2$, then $v$ finishes with $2+2(1/2)=3$.
If $d(v)=3$, then $v$ has no 2-neighbors, so $v$ finishes with 3.
If $d(v)=4$, then $v$ has at most two 2-neighbors, so $v$ finishes with at least
$4-2(1/2)=3$.
If $d(v)=5$, then $v$ has at most four 2-neighbors, so $v$ finishes with at
least $5-4(1/2)=3$.
If $d(v)\ge 6$, then $v$ finishes with at least $d(v)-d(v)/2=d(v)/2\ge 3$, since
$d(v)\ge 6$.  This finishes the proof.
\end{proof}

With a bit more analysis, we can prove the following stronger result.

\begin{cor}
\label{6color-cor}
If $\mad(G)\le 3$ and $G$ does not contain $K^*_7$ as a subgraph, then
$\chio(G)\le 6$.
\end{cor}

\begin{proof}
In the proof of Theorem~\ref{6color-thm}, to contradict the hypothesis
$\mad(G)<3$, we showed that each vertex finished with charge at least 3.
To contradict the weaker hypothesis $\mad(G)\le 3$, it suffices to show that
also some vertex finishes with charge more than 3.
Suppose the contrary.

Now $G$ has no $7^+$-vertex and each 6-vertex has six 2-neighbors.
Each 5-vertex has exactly four 2-neighbors and each 4-vertex has exactly two
2-neighbors.  Now it is straightforward to check that $G$ cannot have any two
adjacent $3^+$-vertices $v$ and $w$; in each case, we delete $v$, $w$, and their
2-neighbors, color the smaller graph and extend.  Thus, $G$ is bipartite, where
one part consists of 2-vertices and the other consists of 6-vertices.  Form $G'$
from $G$ by contracting one edge incident to each 2-vertex.  Now $G'$
is 6-regular, but no component is $K_7$, so $\chi(G')\le 6$ by Brooks' Theorem.
 A 6-coloring of $G'$ induces a 6-coloring of the 6-vertices in $G$, which we
can easily extend to an odd 6-coloring of $G$.
\end{proof}

\section{Odd 5-Colorings}
\label{5coloring-sec}
The goal of this section is to prove the following result.

\begin{thmTb}
\label{5color-thm}
If $\mad(G)<20/7$, then $\chio(G)\le 5$.  
This includes all planar $G$ of girth~at~least~7.
\end{thmTb}
\begin{proof}
The second statement follows from the first by Proposition~\ref{girth-prop}.
So we prove the first.
Assume the statement is false, and $G$ is a counterexample with as few
vertices as possible.  As in the proof of Theorem~\ref{mad4thm}, we assume 
$G$ has no 1-vertex and $G$ has no adjacent 2-vertices.
Suppose $G$ has a 3-vertex $v$ with neighbors $w_1,w_2,w_3$, where
$d(w_1)=d(w_2)=2$.  Denote the other neighbors of $w_1$ and $w_2$ by,
respectively, $x_1$ and $x_2$.  Let $G':=G-\{v,w_1,w_2\}$.
 By minimality, $G'$ has an odd 5-coloring $\vph$.  To extend $\vph$ to $G$,
we color $v$ to avoid $\{\vph(w_1), \vph(w_2), \vph(w_3), 
\vph_o(w_3)\}$ and color each $w_i$ to avoid the new
color on $v$ and avoid $\{\vph(x_i),\vph_o(x_i),\vph(w_3)\}$.  Thus, each
3-vertex has at most one 2-neighbor.  More generally, each 3-vertex has at most
one $3^-$-neighbor.  If not, then we delete $v$ and any 2-neighbors.  We color $v$
to avoid the color on each 3-neighbor, the color on the colored neighbor of each
(deleted) 2-neighbor, as well as $\{\vph(w),\vph_o(w)\}$, where $w$ is the third
neighbor of $v$.  Thus, each 3-vertex has at most one $3^-$-neighbor.

Suppose $G$ has a 4-vertex with four 2-neighbors $w_1,w_2,w_3,w_4$.
For each $i\in \{1,2,3,4\}$, denote the other neighbor of
$w_i$ by $x_i$.  Let $G':=G-\{v,w_1,w_2,w_3,w_4\}$.  By minimality, $G'$ has an odd
5-coloring $\vph$.  
To extend $\vph$ to $G$, color $v$ to avoid
$\{\vph(x_1), \vph(x_2), \vph(x_3)$, $\vph(x_4)\}$ and then
color each $w_i$, in succession, to avoid the new color on $v$ and also avoid
$\{\vph(x_i)$, $\vph_o(x_i), \vph_o(v)\}$.  Thus, no 4-vertex in $G$ has
four 2-neighbors.  Similarly, if each $w_i$ has degree at most 3 and at least
one $w_i$ is a 2-vertex, then we form
$G'$ from $G$ by deleting $v$ and all its 2-neighbors.  Again $G'$ has an odd
5-coloring $\vph$.  We extend $\vph$ to $v$ by avoiding the color on each
3-neighbor and the color on the other neighbor of each 2-neighbor.  Finally, we
can extend the coloring to all 2-neighbors to get an odd 5-coloring of $G$.
Thus, $G$ has no 4-vertex with at least one 2-neighbor and all
$3^-$-neighbors.  Similarly, $G$ does not contain adjacent 4-vertices, $v$ and
$w$, each with three 2-neighbors.  If so, then $G-(N[v]\cup N[w])$ has
an odd 5-coloring.  We color $v$ and $w$ with distinct colors that each differ
from the colors on the three colored vertices at distance two in $G$.  Again,
we can extend this coloring to an odd 5-coloring of $G$.

Now we use discharging to reach a contradiction, which will finish the proof.
We give each vertex $v$ initial charge $d(v)$ and use the following two
discharging rules. 
\begin{enumerate}
\item[(R1)] Each 2-vertex takes $3/7$ from each neighbor.  
\item[(R2)] Each 3-vertex with a 2-neighbor and each 4-vertex with three
2-neighbors takes $1/7$ from each $3^+$-neighbor.  
\end{enumerate}

We show that
each vertex finishes with charge at least 20/7, which contradicts the hypothesis
$\mad(G)<20/7$.  Recall that $\delta(G)\ge 2$.
If $d(v)=2$, then $v$ finishes with $2+2(3/7)=20/7$.
If $d(v)=3$, then $v$ has at most one 2-neighbor.  Further, $v$ does not give
away charge by (R2), since no 3-vertex has both a 2-neighbor and a 3-neighbor,
and also no 4-vertex has both three 2-neighbors and a 3-neighbor.  So, if $v$ has a
2-neighbor, then it has two $4^+$-neighbors and receives charge $1/7$ from each.
Thus, $v$ finishes with at least $3-3/7+2(1/7)=20/7$.  If $v$ has no 2-neighbor,
then $v$ starts and finishes with 3.

Let $v$ be a 4-vertex, and recall that $v$ has at most three 2-neighbors.
If $v$ has at most two 2-neighbors, then $v$ finishes with at least
$4-2(3/7)-2(1/7)=20/7$.
If $v$ has three 2-neighbors, then its fourth neighbor is a $4^+$-neighbor that
does not receive charge from $v$ by (R2) but rather gives $v$ charge $1/7$ by (R2). 
So $v$ finishes with $4-3(3/7)+1/7=20/7$.

If $v$ is a $5^+$-vertex, then $v$ finishes with at least
$d(v)-3d(v)/7=4d(v)/7\ge 20/7$, since $d(v)\ge 5$.
This finishes the proof.
\end{proof}

With a bit more analysis, we can prove the following stronger result.
The proof is similar to that of Corollary~\ref{6color-cor}, so we just provide
a sketch.

\begin{cor}
\label{5color-cor}
If $\mad(G)\le 20/7$ and $G$ does not contain $K^*_6$ as a subgraph, then
$\chio(G)\le 5$.
\end{cor}
\begin{proof}[Proof Sketch.]
Assume $G$ is a counterexample.  So $G$ has no $6^+$-vertices and each 5-vertex
in $G$ has five 2-neighbors.  Each 3-vertex has a single 2-neighbor.
Each 4-vertex either has exactly three 2-neighbors or has exactly two
2-neighbors and gives charge 1/7 to each of it $3^+$-neighbors.  In this case,
we delete $v$, its 4-neighbors and the 2-neighbors of all deleted 4-vertices.
This smaller graph has an odd 5-coloring $\vph$, and it is straightforward to
check that we can extend $\vph$ to $G$.  So $G$ has no 4-vertices.
This implies, in turn, that $G$ has no 3-vertices, since $G$ has no 3-vertex
with two $3^-$-neighbors.  So $G$ is bipartite with vertices in one part of
degree 2 and those in the other part of degree 5.  We contract one edge incident
to each 2-vertex, and 5-color the resulting graph by Brooks' Theorem.
Finally, we extend this 5-coloring to an odd 5-coloring of $G$.
\end{proof}

\bibliographystyle{habbrv}
{{\bibliography{references}}

\end{document}